\documentstyle[11pt]{article}   
\input epsf.sty
\newtheorem{theorem}{Theorem}[section]

\newtheorem{lemma}{Lemma}[section]
\newtheorem{proposition}{Proposition}

\def\RRR{\vbox{\hbox to8.9pt{I\hskip-2.1pt R\hfil}}}
\newcommand{\uno}{\hbox{\bf 1}}

\newcommand{\ds}{\displaystyle}

\newcommand{\barF}{\overline{F}}

\newcommand{\Prob}{\mathop{\rm P}}

\newcommand{\remove}[1]{}

\newcommand{\Reali}{\RRR}

\newcommand{\qed}{\hfill\rule{2mm}{2mm}}  
\def\eq#1{~{\rm (\ref{equation:#1})}}

\newenvironment{proof}{
\begin{trivlist}
\item[\hspace{\labelsep}{\bf\noindent Proof. }]
}{\qed\end{trivlist}}
\def\image #1 by #2 (#3)
{\vbox to #2 {\hrule width #1 height 0pt depth 0pt
\vfill\special{picture #3}}}
\def\scaledimage #1 by #2 (#3 scaled #4)
{\dimen0 = #1\dimen1 = #2
\divide\dimen0 by 1000 \multiply\dimen0 by #4
\divide\dimen1 by 1000 \multiply\dimen1 by #4
\image \dimen0 by \dimen1 (#3 scaled #4)
}
\title{\Large \bf
Simulation of first-passage times\\ 
for alternating Brownian motions
}
\author{
 {\sc  \small{A. DI CRESCENZO$^{(1)}$\footnote{corresponding author}}, }
 {\sc \small{E. DI NARDO$^{(2)}$}, }
 {\sc \small{L.M. RICCIARDI$^{(3)}$}}\\
 \\
 {\small \rm (1) \ Dipartimento di Matematica e Informatica,
 Universit\`a di Salerno} \\
 {\small \rm Via Ponte don Melillo, I-84084 Fisciano (SA), Italy} \\
 {\small \rm E-mail: \tt adicrescenzo@unisa.it}\\
 \\
 {\small \rm (2) \ Dipartimento di Matematica,
 Universit\`a degli Studi della Basilicata}\\
 {\small \rm C.da\ Macchia Romana, I-85100 Potenza, Italy}\\
 {\small \rm E-mail: \tt dinardo@unibas.it}\\
 \\
 {\small \rm (3) \ Dipartimento di Matematica e Applicazioni, 
 Universit\`a di Napoli Federico II}\\
 {\small \rm Via Cintia, I-80126 Napoli, Italy}\\
 {\small\rm E-mail: \tt luigi.ricciardi@unina.it}\\
}
\date{}
\begin{document}
\setlength{\baselineskip}{13pt}
\maketitle
\begin{abstract}\normalsize
The first-passage-time problem for a Brownian motion with alternating infinitesimal 
moments through a constant boundary is considered under the assumption that the time 
intervals between consecutive changes of these moments are described by an alternating 
renewal process. Bounds to the first-passage-time density and distribution 
function are obtained, and a simulation procedure to estimate first-passage-time 
densities is constructed. Examples of applications to problems in environmental 
sciences and mathematical finance are also provided. 

\medskip\noindent
{\bf Keywords:\/} Brownian motion; alternating infinitesimal moments;
renewal process; first-passage time; simulation.

\medskip\noindent
{\bf AMS 2000 Subject Classification:} 
60J65, 
60G40, 
93E30. 
\end{abstract}
%
\section{Introduction}\label{section:A}
Stochastic processes and related first-passage-time (FPT) problems have received  
increasing attention in the literature within the context of model building 
in biology (see, for instance, Ricciardi {\em et al.}, 1999, and references therein). 
Consequently, efforts have also been directed towards the design of computational 
methods suitable to estimate FPT densities, in particular for general Gaussian 
processes (see Di~Nardo {\em et al.}, 2001) and for diffusion processes 
(see references in Ricciardi {\em et al.}, 1999; in particular, 
Buonocore {\em et al.}, 1987, and Giorno {\em et al.}, 1989). 
\par
Here we address the FPT problem through a constant boundary for a model of 
random motion on the real line consisting of a Brownian process whose infinitesimal 
moments alternate between fixed values at the occurrence times of an alternating 
renewal process. Our approach has a twofold aim: (i) to obtain bounds to the FPT 
density and distribution function conditional on the initial infinitesimal moments 
and initial state, and (ii) to set up a simulation procedure for the estimation 
of FPT densities. 
\par
A mathematical specification of the random motion will be given in Section~\ref{section:B}, 
where the key definitions related to the FPT problem through a constant boundary 
will also be recalled. In Section~\ref{section:C}, bounds to the FPT density and 
distribution function will be obtained, and an example of bimodal FPT pdf will be 
provided. In Section~\ref{section:D} a procedure to generate FPTs will be constructed 
by simulating sample paths of the Brownian motion at the random times when the 
infinitesimal moments alternate. Our simulation procedure is specifically based 
on the independent increments property of the Brownian motion and on the 
available closed form expression of its FPT density through a constant boundary. 
In Section~\ref{section:E} two applications to environmental sciences and to 
mathematical finance, respectively, will be provided. In the Appendix, the involved 
probabilistic justification of some crucial steps of the simulation algorithm 
will be presented. 
\par
We wish to mention that Brownian motion with alternating infinitesimal moments 
might provide a fundamental ingredient for the construction of 
a viable model of a phenomenon of great biological relevance: the dynamics of 
the motion of myosin heads along actin filaments, that 
is responsible for the force generation underlying muscle contraction.  
Highly innovative and accurate measurements have indeed proved that such motion consists 
of small steps randomly alternating in direction (see Cyranoski, 2000). 
\section{Brownian motion with alternating behaviors}\label{section:B}
We consider a Brownian motion on the real line characterized by alternating 
behaviors, in the sense that its infinitesimal moments switch between fixed 
values at certain random times. More precisely, starting from state $x_0$ at 
time $0$, the motion undergoes two alternating regimes during random periods 
governed by an alternating renewal process $\{\delta(t),t\geq 0\}$, defined on states $1$ and $2$. 
Here, we assume that $\delta(t)$ is governed 
by two independent sequences of non-negative iid random variables $\{U_1,U_2,\ldots\}$ 
and $\{D_1,D_2,\ldots\}$. We assume that during the periods of random lengths $U_i$, 
$i=1,2,\ldots$, one has $\delta(t)=1$ while the Brownian motion has drift 
$\mu_1\geq 0$ and infinitesimal variance $\sigma_1^2>0$; instead, during the periods 
of random lengths $D_i$, $i=1,2,\ldots$, it is $\delta(t)=2$, and the infinitesimal 
moments take values $\mu_2\leq 0$ and $\sigma_2^2>0$. We assume that at time $0$ the 
initial regime can be arbitrarily specified as $\delta(0)=1$ or $\delta(0)=2$. 
Therefore, the periods of random lengths characterizing the alternating motion 
are $U_1,D_1,U_2,D_2,\ldots$ if $\delta(0)=1$, and $D_1,U_1,D_2,U_2,\ldots$ if 
$\delta(0)=2$. Consequently, for all $t>0$ we have:
$$
 \delta(t)=\left\{
 \begin{array}{ll}
 \ds\sum_{n=1}^{+\infty}\uno_{\{T_{2n-2}\leq t<T_{2n-1}\}}
 +2\ds\sum_{n=1}^{+\infty}\uno_{\{T_{2n-1}\leq t<T_{2n}\}}
 & {\rm if }\; \delta(0)=1, \\
 \hfill & \hfill \\
 2\ds\sum_{n=1}^{+\infty}\uno_{\{T_{2n-2}\leq t<T_{2n-1}\}}
 +\ds\sum_{n=1}^{+\infty}\uno_{\{T_{2n-1}\leq t<T_{2n}\}}
 & {\rm if }\; \delta(0)=2,
 \end{array}
 \right.
$$
where $T_0=0$ and $T_n$, $n=1,2,\ldots$, denotes the $n$-th random instants in 
which the infinitesimal moments change values, with 
\begin{equation}
 T_{2n-1}=T_{2n-2}+\left\{
 \begin{array}{ll}
 U_n & {\rm if }\; \delta(0)=1, \\
 \hfill & \hfill \\
 D_n & {\rm if }\; \delta(0)=2,
 \end{array}
 \right.
 \qquad
 T_{2n}=T_{2n-1}+\left\{
 \begin{array}{ll}
 D_n & {\rm if }\; \delta(0)=1, \\
 \hfill & \hfill \\
 U_n & {\rm if }\; \delta(0)=2.
 \end{array}
 \right.
 \label{equation:29}
\end{equation}
Denoting by $\{X(t),t\geq 0\}$ the stochastic process that describes the motion, 
we thus have: 
\begin{equation}
 X(t)=x_0+\int_0^t \mu(\tau)\,{\rm d}\tau+\sigma(t)\,B(t),
 \qquad t>0.
 \label{equation:1}
\end{equation}
with $X(0)=x_0\in\Reali$,  
\begin{equation}
\mu(t)=\left\{
 \begin{array}{ll}
 \mu_1, & \hbox{if }\delta(t)=1 \\
 \mu_2, & \hbox{if }\delta(t)=2, 
 \end{array}
 \right.
 \qquad 
\sigma(t)=\left\{
 \begin{array}{ll}
 \sigma_1, & \hbox{if }\delta(t)=1 \\
 \sigma_2, & \hbox{if }\delta(t)=2, 
 \end{array}
 \right.
 \label{equation:35}
\end{equation}
and where $\{B(t), t\geq 0\}$ is a 
standard Brownian motion that is assumed to be independent of $\{\delta(t)\}$. 
\par
For all $x\in\Reali$, $t\geq 0$, $x_0\in\Reali$ and $k=1,2$ let us set
$$
 \begin{array}{l}
 f_{k}(x,t\,|\,x_0):=\ds{\partial \over \partial x}
 \Prob\{X(t)\leq x,\delta(t)=1\,|\,X(0)=x_0, \delta(0)=k\}, 
 \vspace{0.1cm} \\
 b_{k}(x,t\,|\,x_0):=\ds{\partial \over \partial x}
 \Prob\{X(t)\leq x,\delta(t)=2\,|\,X(0)=x_0, \delta(0)=k\}.
 \end{array}
$$
Thus, $f_{k}$ is the pdf of $X(t)$ when the drift is positive, while $b_{k}$ relates 
to negative drift, both conditional on initial state $x_0$ and on initial regime $k$. 
Hence, for all $x\in\Reali$ and $t\geq 0$ 
$$
 p_{k}(x,t\,|\,x_0):=f_{k}(x,t\,|\,x_0)+b_{k}(x,t\,|\,x_0)
$$
is the pdf of $X(t)$ conditional on $X(0)=x_0\in\Reali$ and on 
$\delta(0)=k\in\{1,2\}$.\footnote{The determination of $p_k$ in the particular 
case of constant infinitesimal variance and exponentially distributed switching 
times has been provided in Di Crescenzo (2000).} 
\par
Let us now address the FPT problem through a boundary $\beta$ for  
$X(t)$ starting from $x_0\in\Reali$ at time $0$. We assume $\beta>x_0$. (Case $\beta<x_0$ 
can be consequently analyzed by resorting to symmetry considerations). Let 
\begin{equation}
 T_{k}^X=\inf\{t>0: X(t)>\beta\},
 \qquad X(0)=x_0,\; \delta(0)=k\in\{1,2\}
 \label{equation:32}
\end{equation}
be the FPT of $\{X(t)\}$ through $\beta$ from below. We denote by 
\begin{equation}
 H_{k}(\beta,t\,|\,x_0)
 :=\Prob\left\{T_{k}^X\leq t\,\Big|\,X(0)=x_0,\delta(0)=k\right\} 
 \label{equation:23a}
\end{equation}
its cumulative distribution function (cdf) and by 
\begin{equation}
  h_{k}(\beta,t\,|\,x_0)
  :=\ds{\partial \over\partial t}H_{k}(\beta,t\,|\,x_0) 
 \label{equation:23b}
\end{equation}
the corresponding pdf. For $k=1,2$ let now   
\begin{equation}
 W_k(t):=x_0+\mu_k\,t+\sigma_k\,B(t),
 \qquad
 t\geq 0,
 \label{equation:24}
\end{equation}
be the Wiener process characterized by drift $\mu_k\in\Reali$ and infinitesimal
variance $\sigma^2_k>0$, with initial state $W_k(0)=x_0\in\Reali\,$, and 
\begin{equation}
 T_{k}^W=\inf\{t>0: W_k(t)>\beta\},
 \qquad W_k(0)=x_0
 \label{equation:38}
\end{equation}
be the {\em upward\/} FPT of $W_k(t)$ from $x_0$ 
through the boundary $\beta$, with $\beta>x_0$. 
\par
In the foregoing, use of the following formulas will be made: \\ \hfill \\
(i) \ FPT cdf of $W_k(t)$ through $\beta$:
\begin{eqnarray}
 G_{k}(\beta,t\,|\,x_0) \!\!\!\! 
 &=& \!\!\!\! \Prob\left\{T_{k}^W\leq t\,\Big|\,W(0)=x_0\right\} 
 \nonumber \\
 &=& \!\!\!\! \Phi\left(-\ds{\beta-x_0-\mu_k t\over \sigma_k\sqrt{t}}\right)
 +e^{2\mu_k(\beta-x_0)/\sigma_k^2}\,
 \Phi\left(-\ds{\beta-x_0+\mu_k t\over \sigma_k\sqrt{t}}\right),
 \label{equation:10}
\end{eqnarray}
where $\Phi(x)=(2\pi)^{-1/2}\int_{-\infty}^x e^{-y^2/2}\,dy$; note that 
$\ds\lim_{t\to +\infty}G_{k}(\beta,t\,|\,x_0)=1$ for $\mu_k\geq 0$ and $\qquad$
$\ds\lim_{t\to +\infty}G_{k}(\beta,t\,|\,x_0)=\exp\{2\mu_k(\beta-x_0)/\sigma^2_k\}$ 
for $\mu_k\leq 0$; \\ \hfill \\
(ii) \  FPT pdf of $W_k(t)$ through $\beta$:
\begin{equation}
 g_{k}(\beta,t\,|\,x_0)={d\over dt}
 \Prob\left\{T_{k}^W\leq t\,\Big|\,W(0)=x_0\right\}
 ={\beta-x_0\over \sqrt{2\pi\sigma^2_k t^3}}\,
 \exp\left\{-{(\beta-x_0-\mu_k t)^2\over 2\sigma^2_k t}\right\};
 \label{equation:12}
\end{equation}
(iii) \ the $\beta$-avoiding pdf of $W_k(t)$:
\begin{eqnarray}
 && \alpha_{k}(x,t\,|\,x_0) = {\partial\over \partial x}
 \Prob\left\{W(t)\leq x, T_{k}^W>t\,\Big|\,W(0)=x_0\right\}
 \nonumber \\
 && \hspace{0.5cm} 
 ={1\over \sqrt{2\pi\sigma_k^2 t}}\,
 \exp\left\{-{(x-x_0-{\mu_k}t)^2\over 2\sigma^2_k t}\right\}
 \left[1-\exp\left\{-{2\over \sigma^2_k t}\,(\beta-x)\,(\beta-x_0)\right\}\right].
 \label{equation:13}
\end{eqnarray}
%
\section{Bounds to FPT pdf and cdf}\label{section:C}
In this section we obtain bounds to the FPT pdf\eq{23b}
and cdf\eq{23a}. Hereafter we shall denote by
$F_U(t)=\Prob(U_n\leq t)$ and $F_D(t)=\Prob(D_n\leq t)$ the cdf's
of random times $U_n$ and $D_n$, $n=1,2,\ldots$, respectively, and by
$\barF_U(t)=1-F_U(t)$ and $\barF_D(t)=1-F_D(t)$ the corresponding survival
functions.
\subsection{Lower bounds to FPT densities}
First of all we point out the existence of some renewal-type equations involving 
the FPT pdf's $h_k$ defined in\eq{23b}.
\begin{lemma}\label{lemma:1}
For all $t>0$ and $\beta>x_0$ the following equations hold:
\begin{equation}
 h_1(\beta,t\,|\,x_0)=\barF_U(t)\,g_1(\beta,t\,|\,x_0)
 +\ds\int_0^t\int_{-\infty}^{\beta}h_{2}(\beta,t-u\,|\,x)\,
 \alpha_1(x,u\,|\,x_0)\,{\rm d}x\,{\rm d}F_U(u), 
 \label{equation:4a}
\end{equation}
\begin{equation}
 h_{2}(\beta,t\,|\,x_0)=\barF_D(t)\,g_{2}(\beta,t\,|\,x_0)
 +\ds\int_0^t\int_{-\infty}^{\beta}h_1(\beta,t-u\,|\,x)\,
 \alpha_{2}(x,u\,|\,x_0)\,{\rm d}x\,{\rm d}F_D(u). 
 \label{equation:4b}
\end{equation}
\end{lemma}
\begin{proof}
Conditioning on $U_1$, from\eq{23b} we have:
$$
 h_1(\beta,t\,|\,x_0)\,{\rm d}t
 =\int_0^t\Prob\left\{T_1^X\in {\rm d}t\,\Big|\,U_1=u,X(0)=x_0,\delta(0)=1\right\}{\rm d}F_U(u)
 +\barF_U(t)\,g_1(\beta,t\,|\,x_0)\,{\rm d}t.
$$
Eq.\eq{4a} then follows noting that, for $0<u<t$, there results:
$$
 \Prob\left\{T_1^X\in {\rm d}t\,\Big|\,U_1=u,X(0)=x_0,\delta(0)=1\right\}
 =\int_{-\infty}^{\beta}h_{2}(\beta,t-u\,|\,x)\,
 \alpha_1(x,u\,|\,x_0)\,{\rm d}x.
$$
The proof of\eq{4b} can be similarly obtained.
\end{proof}
\par
Lower bounds to FPT densities $h_k$ will now be obtained. 
\begin{theorem}\label{theorem:1}
For all $t>0$ and $\beta>x_0$ we have:
\begin{equation}
 \begin{array}{l}
 h_1(\beta,t\,|\,x_0) 
 \geq \barF_U(t)\,g_1(\beta,t\,|\,x_0)
 + \ds \int_0^t \barF_D(t-u)\,I_1(u,t\,|\,x_0)\,{\rm d}F_U(u), 
 \vspace{0.1cm} \\
 h_{2}(\beta,t\,|\,x_0) 
 \geq \barF_D(t)\,g_{2}(\beta,t\,|\,x_0)
 + \ds \int_0^t \barF_U(t-u)\,I_{2}(u,t\,|\,x_0)\,{\rm d}F_D(u).
 \end{array}
 \label{equation:11}
\end{equation}
where, for $0<u<t$, 
\begin{equation}
 \begin{array}{l}
 I_1(u,t\,|\,x_0):=\ds\int_{-\infty}^{\beta}g_{2}(\beta,t-u\,|\,x)\,
 \alpha_1(x,u\,|\,x_0)\,{\rm d}x\,,
 \vspace{0.1cm} \\
 \hspace{2.1cm} =\ds\frac{1}{\sqrt{2 \pi [u \sigma_1^2+ (t-u)\sigma_2^2]^3}}
 \left[A_{-}(t,u)+A_{+}(t,u)\right]\,,
 \vspace{0.1cm} \\
 I_2(u,t\,|\,x_0):=\ds\int_{-\infty}^{\beta}g_1(\beta,t-u\,|\,x)\,
 \alpha_{2}(x,u\,|\,x_0)\,{\rm d}x\,,
 \vspace{0.1cm} \\
 \hspace{2.3cm} =\ds\frac{1}{\sqrt{2 \pi [u \sigma_2^2+ (t-u)\sigma_1^2]^3}}
 \left[B_{-}(t,u)+B_{+}(t,u)\right]\,,
 \end{array}
 \label{equation:27}
\end{equation}
with 
\begin{eqnarray*}
 && A_{\pm}(t,u) = [(\beta-x_0)\sigma_2^2\pm u(\mu_1\sigma_2^2-\mu_2\sigma_1^2)] \\
 && \hspace{1.6cm} 
 \times \exp\left\{{-\frac{[-\mu_1 u-\mu_2(t-u)\mp(\beta-x_0)]^2}
 {2 [u \sigma_1^2+ (t-u) \sigma_2^2]}}\right\}
 \exp\left\{-\frac{(\mu_1 \pm \mu_1)(\beta-x_0)}{\sigma_1^2}\right\} \\
 && \hspace{1.6cm} 
 \times \Phi\left(- [u(\mu_1\sigma_2^2-\mu_2 \sigma_1^2) \pm (\beta-x_0)\sigma_2^2] 
 \sqrt{\frac{t-u}{\sigma_1^2\sigma_2^2u[u\sigma_1^2+(t-u)\sigma_2^2]}}\right)
\end{eqnarray*}
and 
\begin{eqnarray*}
 && B_{\pm}(t,u) = [(\beta-x_0)\sigma_1^2\pm u(\mu_1\sigma_2^2-\mu_2\sigma_1^2)] \\
 && \hspace{1.6cm} 
 \times \exp\left\{{-\frac{[-\mu_1(t-u)-\mu_2 u\pm(\beta-x_0)]^2}
 {2[u \sigma_2^2+ (t-u) \sigma_1^2]}} \right\}
 \exp\left\{\frac{(\mu_2 \mp \mu_2)(\beta-x_0)}{\sigma_2^2}\right\} \\
 && \hspace{1.6cm} 
 \times \Phi\left([u(\mu_1 \sigma_1^2 -\mu_2\sigma_2^2) \pm (\beta-x_0) \sigma_1^2]
 \sqrt{\frac{t-u}{\sigma_1^2 \sigma_2^2 u[u \sigma_2^2+ (t-u)\sigma_1^2]}}\right).
\end{eqnarray*}
\end{theorem}
\begin{proof}
Substituting\eq{4b} in\eq{4a}, for all $t>0$ and $\beta>x_0$ we have 
\begin{equation}
\begin{array}{l}
 h_1(\beta,t\,|\,x_0)
 =\barF_U(t)\,g_1(\beta,t\,|\,x_0)
 +\ds\int_0^t\int_{-\infty}^{\beta}\barF_D(t-u)\,g_{2}(\beta,t-u\,|\,x)\,
 \alpha_1(x,u\,|\,x_0)\,{\rm d}x\,{\rm d}F_U(u)
 \\
 +\ds\int_0^t\int_{-\infty}^{\beta}\left[
 \int_0^{t-u}\int_{-\infty}^{\beta}h_1(\beta,t-u-v\,|\,y)\,
 \alpha_{2}(y,v\,|\,x)\,{\rm d}y\,{\rm d}F_D(v)\right] 
 \alpha_1(x,u\,|\,x_0)\,{\rm d}x\,{\rm d}F_U(u).
 \end{array}
 \label{equation:5a}
\end{equation}
Similarly, substitution of\eq{4a} in\eq{4b} yields:
\begin{equation}
\begin{array}{l}
  h_{2}(\beta,t\,|\,x_0)
 =\barF_D(t)\,g_{2}(\beta,t\,|\,x_0)
 +\ds\int_0^t\int_{-\infty}^{\beta}\barF_U(t-u)\,g_1(\beta,t-u\,|\,x)\,
 \alpha_{2}(x,u\,|\,x_0)\,{\rm d}x\,{\rm d}F_D(u)
 \\
 +\ds\int_0^t\int_{-\infty}^{\beta}\left[
 \int_0^{t-u}\int_{-\infty}^{\beta}h_{2}(\beta,t-u-v\,|\,y)\,
 \alpha_1(y,v\,|\,x)\,{\rm d}y\,{\rm d}F_U(v)\right] 
 \alpha_{2}(x,u\,|\,x_0)\,{\rm d}x\,{\rm d}F_D(u). 
 \end{array}
 \label{equation:5b}
\end{equation}
Hence, from Eqs.\eq{5a} and\eq{5b} we obtain 
\begin{eqnarray*}
 && \hspace{-0.6cm}
 h_1(\beta,t\,|\,x_0) \geq \barF_U(t)\,g_1(\beta,t\,|\,x_0)
 + \int_0^t\int_{-\infty}^{\beta}\barF_D(t-u)\,g_{2}(\beta,t-u\,|\,x)\,
 \alpha_1(x,u\,|\,x_0)\,{\rm d}x\,{\rm d}F_U(u), \\
 && \hspace{-0.6cm}
 h_{2}(\beta,t\,|\,x_0) \geq \barF_D(t)\,g_{2}(\beta,t\,|\,x_0)
 + \int_0^t\int_{-\infty}^{\beta}\barF_U(t-u)\,g_1(\beta,t-u\,|\,x)\,
 \alpha_{2}(x,u\,|\,x_0)\,{\rm d}x\,{\rm d}F_D(u).
\end{eqnarray*}
Recalling\eq{12} and\eq{13}, after some calculations Eqs.\eq{11} follow. 
\end{proof}
\par
We point out the following probabilistic interpretation of the functions defined 
in\eq{27}: for $k=1,2$, $I_k(u,t\,|\,x_0)$ identifies with the FPT density through $\beta$, 
when the initial infinitesimal moments are $\mu_k$ and $\sigma^2_k$, with initial 
state $x_0$, jointly with the condition that up to time $t$ only one inversion of 
infinitesimal moments has occurred, such an inversion having occurred at time $u$. 
\subsection{Bounds to FPT distribution functions}
We shall now obtain upper and lower bounds to the FPT cdf defined in\eq{23a}. 
To this purpose we prove that two integral equations for $H_1$ and $H_2$ hold. 
\begin{theorem}\label{lemma:2}
For all $t>0$ and $\beta>x_0$ the following equalities hold:
\begin{eqnarray}
 && \hspace{-1.5cm} H_1(\beta,t\,|\,x_0)
 =\barF_U(t)\,G_1(\beta,t\,|\,x_0)
 +\int_0^t G_1(\beta,u\,|\,x_0)\,{\rm d}F_U(u)
 \nonumber \\
 && \;+\int_0^t {\rm d}F_U(u)\int_{-\infty}^{\beta}
 \Bigg\{\barF_D(t-u)\,G_{2}(\beta,t-u\,|\,x)
 +\int_0^{t-u} G_{2}(\beta,v\,|\,x)\,{\rm d}F_D(v)
 \nonumber \\
 && \;+\int_0^{t-u}{\rm d}F_D(v)\int_{-\infty}^{\beta}\alpha_{2}(y,v\,|\,x)\,
 H_1(\beta,t-u-v\,|\,y)\,{\rm d}y\Bigg\}\alpha_1(x,u\,|\,x_0)\,{\rm d}x,
 \label{equation:7a}
 \\
 && \hspace{-1.5cm} H_{2}(\beta,t\,|\,x_0)
 =\barF_D(t)\,G_{2}(\beta,t\,|\,x_0)
 +\int_0^t G_{2}(\beta,u\,|\,x_0)\,{\rm d}F_D(u)
 \nonumber \\
 && \;+\int_0^t {\rm d}F_D(u)\int_{-\infty}^{\beta}
 \Bigg\{\barF_U(t-u)\,G_1(\beta,t-u\,|\,x)
 +\int_0^{t-u} G_1(\beta,v\,|\,x)\,{\rm d}F_U(v)
 \nonumber \\
 && \;+\int_0^{t-u}{\rm d}F_U(v)\int_{-\infty}^{\beta}\alpha_{1}(y,v\,|\,x)\,
 H_{2}(\beta,t-u-v\,|\,y)\,{\rm d}y\Bigg\}\alpha_{2}(x,u\,|\,x_0)\,{\rm d}x.
 \label{equation:7b}
\end{eqnarray}
\end{theorem}
\begin{proof}
Conditioning on $U_1$, from Eq.\eq{23a} we have:
\begin{equation}
 H_1(\beta,t\,|\,x_0)
 =\int_0^t\Prob\left(T_1^X\leq t\,
 \Big|\,U_1=u,X(0)=x_0,\delta(0)=1\right)\,{\rm d}F_U(u)
 +\barF_U(t)\,G_1(\beta,t\,|\,x_0).
 \label{equation:39}
\end{equation}
Hence, by noting that for $0<u<t$ there holds
\begin{eqnarray}
 \Prob\left(T_1^X\leq t\,|\,U_1=u\right) \!\!\!\!
 &=& \!\!\!\! \Prob\left(T_1^X\leq u\,|\,U_1=u\right)
 +\Prob\left(u<T_1^X\leq t\,|\,U_1=u\right) 
 \nonumber \\
 &=& \!\!\!\! G_1(\beta,u\,|\,x_0)
 +\int_{-\infty}^{\beta}\alpha_1(x,u\,|\,x_0)\,H_{2}(\beta,t-u\,|\,x)\,{\rm d}x 
 \label{equation:40}
\end{eqnarray}
one obtains: 
\begin{eqnarray}
 H_1(\beta,t\,|\,x_0) \!\!\!\!
 &=& \!\!\!\! \barF_U(t)\,G_1(\beta,t\,|\,x_0)
 +\int_0^t G_1(\beta,u\,|\,x_0)\,{\rm d}F_U(u)
 \nonumber \\
 &+& \!\!\!\! \int_0^t \int_{-\infty}^{\beta}
 \alpha_1(x,u\,|\,x_0)\,H_{2}(\beta,t-u\,|\,x)\,{\rm d}x\,{\rm d}F_U(u). 
 \label{equation:6a} 
\end{eqnarray}
By a similar argument, one is led to the following equation: 
\begin{eqnarray}
 H_2(\beta,t\,|\,x_0) \!\!\!\!
 &=& \!\!\!\! \barF_D(t)\,G_{2}(\beta,t\,|\,x_0)
 +\int_0^t G_{2}(\beta,u\,|\,x_0)\,{\rm d}F_D(u)
 \nonumber \\
 &+& \!\!\!\! \int_0^t \int_{-\infty}^{\beta}
 \alpha_{2}(x,u\,|\,x_0)\,H_1(\beta,t-u\,|\,x)\,{\rm d}x\,{\rm d}F_D(u).
 \label{equation:6b}
\end{eqnarray}
Substitution of the right-hand-side of\eq{6b} into\eq{6a}, and viceversa, finally 
yields Eqs.\eq{7a} and\eq{7b}.
\end{proof}
\par
The following theorem provides upper and lower bounds to the
FPT cdf's $H_1$ and $H_2$.
\begin{theorem}\label{theorem:2}
For all $t>0$ and $\beta>x_0$ there holds:
\begin{eqnarray}
 && \hspace{-1cm}
 \barF_U(t)\,G_1(\beta,t\,|\,x_0)
 +\int_0^t G_1(\beta,u\,|\,x_0)\,{\rm d}F_U(u)
 \leq H_1(\beta,t\,|\,x_0), 
 \label{equation:9a} \\
 && \hspace{-1cm}
 \barF_D(t)\,G_{2}(\beta,t\,|\,x_0)
 +\int_0^t G_{2}(\beta,u\,|\,x_0)\,{\rm d}F_D(u)
 \leq H_{2}(\beta,t\,|\,x_0).
 \label{equation:9b}
\end{eqnarray}
Moreover, if $\sigma_1=\sigma_2$, for all $k\in\{1,2\}$, $t>0$ and $\beta>x_0$ it is:
\begin{equation}
 H_k(\beta,t\,|\,x_0)\leq G_1(\beta,t\,|\,x_0). 
 \label{equation:37}
\end{equation}
\end{theorem}
\begin{proof}
Inequalities\eq{9a} and\eq{9b} are a consequence of\eq{7a} and\eq{7b}, respectively. 
To prove the validity\eq{37}, we notice that by virtue of\eq{23a} and\eq{10}, the 
upper bound given in\eq{37} can be rewritten as 
\begin{equation}
 T_{1}^W\leq_{\rm st}T_{k}^X \qquad \hbox{for all $k\in\{1,2\}$,}
 \label{equation:28}
\end{equation}
where $T_{1}^W$ and $T_{k}^X$ are defined in\eq{38} and\eq{32}, and where 
$\leq_{\rm st}$ denotes the customary stochastic order (see Section 1.A of 
Shaked and Shanthikumar, 1994). Condition\eq{28} can be seen to follow from 
$\{X(t),t\geq 0\}\leq_{\rm st}\{W_1(t),t\geq 0\}$, which can be obtained by 
construction of processes ``clones'' of $\{X(t),t\geq 0\}$ and $\{W_1(t),t\geq 0\}$ 
on the same probability space (see Section 4.7.B of Shaked and Shanthikumar, 1994), and 
by recalling definitions\eq{1} for $\sigma(t)=\sigma_1=\sigma_2$ and\eq{24} for $k=1$. 
\end{proof}
%
\section{Simulation of FPTs}\label{section:D}
In the previous section we have obtained bounds to the FPT pdf and cdf. These are 
of interest because closed-form expressions are not known. In order to obtain 
estimates of the FPT pdf $h_{k}(\beta,t\,|\,x_0)$, in this section we construct 
a numerical procedure\footnote{Use of such procedure was made in a previous paper 
(Buonocore {\em et al.}, 2001) without disclosing the theoretical arguments on 
which it rests and without explaining how to perform certain critical steps.} 
for the simulation of the FPT of $\{X(t)\}$ through a constant boundary $\beta$. 
The values of $X(t)$ are simulated only at the switching times of the infinitesimal 
moments. Our simulation is based on the following equations holding for $x_0<\beta$ 
and $k\in\{1,2\}$: 
\begin{eqnarray}
 && \hspace{-1.5cm}
 \Prob\left\{T_{k}^X\leq t\,\Big|\,X(0)=x_0,\delta(0)=k, T_1=\tau\right\}
 =\Prob\left\{T_{k}^W\leq t\,\Big|\,W_k(0)=x_0\right\},
 \quad 0<t\leq \tau, 
 \label{equation:25} \\
 && \hspace{-1.5cm}
 \Prob\left\{X(\tau)\leq x\,\Big|\,X(0)=x_0,\delta(0)=k, T_1=\tau, T_{k}^X>\tau\right\} 
 \nonumber \\
 && \hspace{3cm}
 =\Prob\left\{W_k(\tau)\leq x\,\Big|\,W_k(0)=x_0,T_{k}^W>\tau\right\},
 \quad \tau>0, \;\; x\leq \beta.
 \label{equation:26}
\end{eqnarray}
In other words, during each random period $(T_n, T_{n+1}]$, $n=0,1,\ldots$, 
$X(t)$ behaves as a Wiener process $W_k(t)$ with fixed infinitesimal moments, so that 
the probabilities on the left-hand-sides of (\ref{equation:25}) and (\ref{equation:26}) 
can be expressed in terms of the corresponding probabilities 
for $W_k(t)$, as indicated in the right-hand-sides of (\ref{equation:25}) 
and (\ref{equation:26}). 
\par
The simulation procedure is specified as follows. 
\par 
$$
 \framebox{
 $\begin{array}{l} 
 \hbox{\bf The main simulation procedure}\\
 \hbox{\textsc{Step 0.}$\;$ input($\mu_1$, $\mu_2$, $\sigma_1$, $\sigma_2$, $k$, $\beta$, $x_0$, $t_{max}$, $\varepsilon_1$);}\\
 \hbox{\textsc{Step 1.}$\;$ $x:=x_0$, $t:=0$;}\\
 \hbox{\textsc{Step 2.}$\;$ if $k=1$ then \{$\mu_k:= \mu_1$, $\sigma_k:= \sigma_1$, $F:=F_U$\}, 
       else \{$\mu_k:= \mu_2$, $\sigma_k:= \sigma_2$, $F:=F_D$\};}\\
 \hbox{\textsc{Step 3.}$\;$ if $t>t_{max}$ then stop;}\\
 \hbox{\textsc{Step 4.}$\;$ generate the inversion instant $\tau$ according to distribution $F$;}\\
 \hbox{\textsc{Step 5.}$\;$ $p:= G_{k}(\beta,\tau\,|\,x)$; \quad [see Eq.\eq{10}]}\\
 \hbox{\textsc{Step 6.}$\;$ if $p<\varepsilon_1$ then goto \textsc{Step~12};}\\
 \hbox{\textsc{Step 7.}$\;$ generate an uniform pseudo-random number $u$ in $(0,1)$;}\\
 \hbox{\textsc{Step 8.}$\;$ if $u>p$ then goto \textsc{Step~12};}\\
 \hbox{\textsc{Step 9.}$\;$ generate a 
       pseudo-random number $\theta$ in $(0,\tau)$ from pdf $f_1(\theta)$; \quad [see Eq.\eq{14}]}\\
 \hbox{\textsc{Step 10.}$\;$ $fpt:=\theta+t$;}\\
 \hbox{\textsc{Step 11.}$\;$ output($fpt$); stop;}\\
 \hbox{\textsc{Step 12.}$\;$ generate a 
       pseudo-random number $z$ in $(-\infty,\beta)$ from pdf $f_2(z)$; \quad [see Eq.\eq{15}]}\\
 \hbox{\textsc{Step 13.}$\;$ $x:=z$, $t:=t+\tau$;}\\
 \hbox{\textsc{Step 14.}$\;$ if $k=1$ then $k:=2$, else $k:=1$;}\\
 \hbox{\textsc{Step 15.}$\;$ goto \textsc{Step~3}.}\\
 \end{array}$ 
} 
$$
\par
The simulation starts from $x=x_0$ at time $t=0$, where $x$ and $t$ denote current 
state and current time of $X(t)$, respectively. The integral variable $k$ specifies 
the regime of the process at time $0$: $k=1$ if the infinitesimal moments are 
$\mu_1$ and $\sigma_1^2$, whereas $k=2$ for $\mu_2$ and $\sigma_2^2$. The first 
switching instant $\tau$ of the infinitesimal moments is generated according 
to distribution $F_U$ if $k=1$, or according to $F_D$ if $k=2$. 
The simulation then proceeds by generating a 
Bernoulli random number $u$ and by comparing it (Step 8) with the FPT 
probability $G_{k}(\beta,\tau\,|\,x)$ given by\eq{10}. In order to avoid 
certain numerical problems, when $G_{k}(\beta,\tau\,|\,x_0)$ is smaller 
than a tolerance parameter $\varepsilon_1$, the procedure acts like as if 
$G_{k}(\beta,\tau\,|\,x_0)$ is zero. 
\par
(i) \ If $u\leq G_{k}(\beta,\tau\,|\,x)$, the simulated sample-path of
$X(t)$ crosses the boundary $\beta$ at a suitably generated crossing 
time $\theta\in (0,\tau)$. Note that, since  crossing occurs before 
inversion-time $\tau$, the alternation of infinitesimal moments has no 
effect on the determination of $\theta$. Such instant $\theta$ has to be 
simulated from the FPT pdf $f_1(\theta)$ conditioned by $T_{k}^W\leq\tau$
and $W_k(0)=x$. Due to\eq{10} and\eq{12}, such density is  
\begin{equation}
 f_1(\theta)={g_{k}(\beta,\theta\,|\,x)
 \over G_{k}(\beta,\tau\,|\,x)}.
 \qquad 0<\theta<\tau.
 \label{equation:14}
\end{equation}
After the simulation of the first crossing instant $\theta$ by the method 
described in Section A1, the procedure ends.
\par
(ii) \ If $u>G_{k}(\beta,\tau\,|\,x_0)$, the first passage has not occurred
before time $\tau$. The state $X(\tau)=z\in(-\infty,\beta)$ is then generated 
by means of the method shown in Section A2. Such state is simulated from 
the r.v.\ $W_k(\tau)$ conditioned by $T_{k}^W>\tau$ and $W_k(0)=x$. 
Due to Eqs.\eq{12} and\eq{13} its pdf is given by
\begin{equation}
 f_2(z)={\alpha_{k}(z,\tau\,|\,x)
 \over 1-G_{k}(\beta,\tau\,|\,x)},
 \qquad -\infty<z<\beta.
 \label{equation:15}
\end{equation}
The values of $x$, $t$ and $k$ are then restored: the current state $x$ is set 
to $z$, the current time $t$ is set to $t+\tau$ and the infinitesimal moments 
are switched, in the sense that the new value of $k$ is set equal to $3-k$.
A key feature of this procedure is that, due to the property of independent 
increments of the Brownian motion, the switching instants of the infinitesimal 
moments are regenerative. The simulation thus proceeds with the generation of a new
inversion instant, assuming $x$ as the new initial state at the current time $t$.
The procedure goes on until the first passage has occurred, or until a
preassigned maximum time $t_{max}$ is reached.
\par
By means of the described procedure a random sample of first-passage instants is 
obtained to construct estimates of the unknown FPT pdf. To this purpose, the 
following kernel estimator for $h_{k}(\beta,t\,|\,x_0)$ has been employed:
\begin{equation}
 \hat{h}(t)=\frac{1}{n\,\Delta}\,\sum_{i=1}^n K\!\left( \frac{t-T_i}{\Delta}\right), 
 \quad t > 0,
 \label{equation:36}
\end{equation}
where $\{T_i; \;i=1,2,\ldots,n\}$ is the random sample obtained by repeated 
use of the simulation procedure, $\Delta$ is the bandwidth of the kernel, and
$$
 K(t)
 :=\left\{ \begin{array}{ll}
 \ds{\frac{3}{4\sqrt{5}}\,\left(1-\frac{t^2}{5}\right)}, & t \in [-\sqrt{5},\sqrt{5}] \\
 0,                   & \hbox{otherwise}
 \end{array} 
 \right.
$$
is the Epanechnikov kernel (see Silverman, 1986). Estimates for 
$H_{k}(\beta,t\,|\,x_0)$ can thus be obtained by numerical integration of\eq{36}. 
\par
Some applications of the foregoing results and the simulation procedure 
will be discussed in Section \ref{section:E}, whereas a detailed 
description of the simulations of the random variables characterized by 
pdf's\eq{14} and\eq{15} will be given in the Appendix. 
\par
We point out that the role of alternating processes governed by 
exponentially-distributed alternating times has already been outlined in the 
biomathematical literature. For instance, as mentioned in Stadje (1987), some 
experimental studies suggest that the motion of certain micro-organisms can 
be approximated by trajectories that change directions at exponentially 
distributed random times. Another example of great biological interest, 
already referred to in Section 1, is found in Kitamura {\em et al.} (1999), 
where it is experimentally shown that the dwell times between consecutive 
jumps in the rising phase of myosin movements along actin filament during 
muscle contraction are exponentially distributed. Furthermore, aiming to 
describe the firing activity of a neuronal model subject to alternating 
input, a Wiener process with drift alternating at exponentially distributed 
times was recently studied by Buonocore {\em et al.} (2001). 
These are paradigmatic examples that lead one to focus attention on Brownian 
motion with alternating infinitesimal moments in the special case when the 
alternating times are exponentially distributed. Hereafter we show that in 
such a case the FPT pdf through constant boundaries for process\eq{1} is not 
necessarily unimodal, quite differently from the case of the Wiener process that 
leads to rigorously unimodal densities. Figure 1 shows\footnote{Throughout this paper, 
all numerical results are obtained by using a simulated random sample of size $10^6$.} 
a bimodal estimate of FPT pdf $h_{2}(3,t\,|\,0)$ obtained via our simulation 
procedure, together with the corresponding lower bound provided in\eq{11}, 
under the assumption that the alternating random times are exponentially 
distributed, with $\barF_U(t)=e^{-t/2}$ and $\barF_D(t)=e^{-2\,t}$, $t\geq 0$. 
%
\section{Applications}\label{section:E}
Hereafter, we shall indicate two different applications of the foregoing results and 
simulation procedure. 
The first application is of interest to environmental sciences. Aiming to theoretically 
construct a link between the intermittence of rainfall and the dynamics of moisture 
processes, in Freidlin and Pavlopoulos (1997) a stochastic model has been proposed 
in which the temporal evolution of the moisture content in a given atmospheric column 
is described by a stochastic process. This consists of the alternation of two Wiener 
processes with drifts of opposite signs and unequal infinitesimal variances, the 
alternation taking place when an upper saturation threshold or a lower dehydration 
threshold is reached. As a possible alternative model, we select process\eq{1} to 
describe the temporal evolution of the moisture content. Differently from 
Freidlin and Pavlopoulos (1997), in our model the alternation between two different 
regimes does not occur at the thresholds' reaching times, but at the occurrence 
of random times\eq{29}. However, in order to keep a similarity between the 
alternating mechanism of those two models, we assume that the inter-swithching 
times are distributed as in Freidlin and Pavlopoulos (1997). Hence, random 
variables $U_1,U_2,\ldots$ and $D_1,D_2,\ldots$ are now assumed to have 
inverse Gaussian pdf's 
\begin{equation}
	f_U(x)=\sqrt{l_1\over 2\pi x^3}\,\exp\left\{-{l_1(x-m_1)^2\over 2m^2_1 x}\right\}, 
 \qquad 
	f_D(x)=\sqrt{l_2\over 2\pi x^3}\,\exp\left\{-{l_2(x-m_2)^2\over 2m^2_2 x}\right\}, 
 \label{equation:30}
\end{equation}
for $x>0$. The removal of the thresholds is acceptable because, according to 
Freidlin and Pavlopoulos (1997), no empirical verification is available for their 
existence. Furthermore, it is also reasonable to assume that the alternating of dry 
and wet durations is regulated by a mechanism somewhat looser than that including 
the presence of precisely specified saturation and dehydration thresholds. 
\par
As study case we refer to the measurements performed within the TOGA-COARE 
(Tropical Ocean Global Atmosphere -- Coupled Ocean-Atmosphere Response Experiment) 
and reported in Freidlin and Pavlopoulos (1997). Making use of the estimates obtained 
by these authors via the method of moments, we are led to the following values 
for the parameters appearing in\eq{30} and for the alternating infinitesimal 
moments of $\{X(t)\}$:
\begin{equation}
 \begin{array}{llll}
	l_1=0.7518,  &  m_1=1.4215, 
 &
 l_2=0.5073,  &  m_2=1.0476, \\
 \mu_1=0.2313,  &   \sigma_1=0.3792,
 & 
 \mu_2=-0.3139, &   \sigma_2=0.4616.
\end{array}
\label{equation:31}
\end{equation}
Since the magnitude of negative drift exceeds that of positive drift, a 
negative net displacement is obtained. Figure 2 shows the estimates of FPT pdf 
$h_{k}(\beta,t\,|\,x_0)$ and of the corresponding cdf $H_{k}(\beta,t\,|\,x_0)$ 
obtained by means of our simulation procedure, as well as the corresponding 
lower bounds provided in Section~\ref{section:C}. Computations make use of 
densities\eq{30} and of the estimates\eq{31} of parameters. The simulation 
of the inverse Gaussian distributed random times has been performed by 
the method of Michael {\em et al.} (1976).
\par
We now come to an application in mathematical finance. As is well known, one 
of stochastic processes often used to describe the time course of the price 
$S(t)$ of risky assets is the geometric Brownian motion
$$
 S(t)=s_0\,\exp\{\mu\,t+\sigma\,B(t)\}, \qquad t\geq 0.
$$
An extension of this model is based on the assumption that parameters $\mu$ and 
$\sigma$ alternate between two values, with variations occurring randomly in time 
according to an alternating renewal process. Such alternation is meant to be 
responsible for the floating behavior of prices, which are often subject to 
alternating periods of growing and decreasing trends, though maintaining their 
high level of stochasticity. We thus refer to a continuous-time stochastic 
process $\{S(t),t\geq 0\}$ with state-space $(0,+\infty)$ defined as
\begin{eqnarray}
 S(t) \!\!\! &=& \!\!\! s_0\,\exp\{X(t)-x_0\} 
 \nonumber \\
 &=& \!\!\! s_0\,\exp\left\{\int_0^t \mu(\tau)\,{\rm d}\tau+\sigma(t)\,B(t)\right\},
 \qquad t\geq 0,
 \label{equation:34}
 \end{eqnarray}
with $s_0>0$, and where $X(t)$ is defined in\eq{1}, and $\mu(t)$ and $\sigma(t)$ 
are given in\eq{35}. Model\eq{34} characterizes a growing trend when $\delta(t)=1$, 
i.e.\ $\mu(t)=\mu_1\geq 0$ and $\sigma(t)=\sigma_1$, and a decreasing trend 
when $\delta(t)=2$, i.e.\ $\mu(t)=\mu_2\leq 0$ and $\sigma(t)=\sigma_2$. 
The initial trend is determined by $\delta(0)\in\{1,2\}$. 
\par
To come to the FPT problem through the boundary $\beta>s_0$, let 
$$
 T_{k}^S:=\inf\{t>0: S(t)>\beta\},
 \qquad S(0)=s_0,\; \delta(0)=k\in\{1,2\}.
$$
Hence, $T_{k}^S$ describes the first time when price $S(t)$ reaches 
the level $\beta$ from below. This is of interest in mathematical finance  
when $\beta$ may represents some critical threshold for price $S(t)$. 
We denote by $h_k^S(\beta,t\,|\,s_0)$ and by $H_k^S(\beta,t\,|\,s_0)$ the pdf and 
the cdf of $T_{k}^S$, respectively. Recalling\eq{23a} and\eq{23b}, from\eq{34} 
and the independence of the increments of $\{X(t)\}$ one can see that 
\begin{equation}
  h_k^S(\beta,t\,|\,s_0)=h_k\left(\ln{\beta\over s_0},t\,\Bigg|\,0\right),
 \qquad
  H_k^S(\beta,t\,|\,s_0)=H_k\left(\ln{\beta\over s_0},t\,\Bigg|\,0\right),
 \qquad 
 t>0.
 \label{equation:33}
 \end{equation}
Hence, $h_k^S$ and $H_k^S$ can be estimated by resorting to the method 
described in Section \ref{section:D} for $h_k$ and $H_k$, whereas bounds for $h_k^S$ and 
$H_k^S$ can be obtained by means of Eq.\eq{33} and Theorems \ref{theorem:1} and \ref{theorem:2}. 
\par
In order to include in model\eq{34} the occurrence of periods of alternating trends 
characterized by heavy-tailed distributions, we assume that the random times 
$U_n$, $D_n$ have respectively generalized Pareto distribution 
\begin{equation}
 F_U(x)=1-\left(1+\xi_1\,\frac{x}{\eta_1}\right)^{-1/\xi_1}, 
 \qquad 
 F_D(x)=1-\left(1+\xi_2\,\frac{x}{\eta_2}\right)^{-1/\xi_2}, 
 \qquad x>0,
 \label{equation:41}
 \end{equation}
with $\xi_k>0$ and $\eta_k>0$, $k=1,2$. This is quite reasonable, since such distribution 
is known to play a relevant role in various applied fields such as those related to 
insurance, finance, telecommunication traffic, queueing, in which (see, for instance, 
Mikosch and Nagaev, 1998, and Greiner {\em et al.\/}, 1999) measurements of time series 
show large degrees of variability in the arrival rates, and long-range dependence effects. 
\par
An example of estimate of FPT pdf $h_{1}^S(2,t\,|\,s_0)$ and of its lower bound, 
obtained by means of our simulation procedure is provided in Figure 3, 
whereas the corresponding cdf $H_{1}^S(2,t\,|\,s_0)$ and its lower and 
upper bounds are shown in Figure 4. 
\section*{Appendix} 
This Appendix is devoted to a thorough description of the probability arguments 
underlying some steps of the procedure referred to in Section \ref{section:D} 
to simulate the random variables having pdf's\eq{14} and\eq{15}. 
\subsection*{A1. Simulation when first passage has already occurred} 
Based on the Von~Neumann acceptance-rejection (VNAR) method (see for instance 
Ross, 1989) the following procedure simulates the generation of a pseudo-random 
number from the r.v.\ with pdf $f_1(\theta)$ given in\eq{14}: 
$$
 \framebox{
 $\begin{array}{l}
 \hbox{\bf Procedure}\\
 \hbox{\textsc{Step 1.}$\;$ evaluate the maximum $m:=f_1(\theta_m)$ of the pdf\eq{14};}\\
 \hbox{\textsc{Step 2.}$\;$ if $1/(m \tau) \leq \varepsilon_2$ then $\theta:= \theta_m$ and goto \textsc{Step 5};}\\
 \hbox{\textsc{Step 3.}$\;$ generate two uniform independent pseudo-random numbers $x$ in $(0,\tau)$ and $y$ in $(0,m)$;}\\
 \hbox{\textsc{Step 4.}$\;$ if $y \leq f_1(x)$ then $\theta:=x$, else goto \textsc{Step 3};}\\
 \hbox{\textsc{Step 5.}$\;$ return($\theta$).} 
 \end{array}$ 
} 
$$
Concerning Step 1, we point out that density\eq{14} is limited by the maximum 
$m=f_1(\theta_m)$ located at $\theta_m=\min\{\tau, c_m\}$, where
$$
 c_m=\cases{
 \ds{-3\sigma^2_k+\sqrt{9\sigma^4_k+4\mu_k^2(\beta - x)^2}\over 2\mu_k^2}
 & if $\mu_k\neq 0$,\cr
 \hfill\cr
 \ds{(\beta-x)^2\over 3\sigma^2_k}
 & if $\mu_k= 0$.}
$$
Then, $m\tau\geq\int_0^{\tau}f_1(\theta)\,d\theta=1$ since $m\geq f_1(\theta)$ for all 
$\theta\in(0,\tau)$. Hence, by Step 1 the maximum $m:=f_1(\theta_m)$ of the pdf\eq{14} 
is evaluated, to yield $\theta_m$ as output if $1/(m \tau) \leq \varepsilon_2$, where 
$\varepsilon_2$ is a preassigned tolerance parameter. If $1/(m \tau)>\varepsilon_2$, by 
Step 3 two uniform independent pseudo-random numbers $x$ in $(0,\tau)$ and $y$ in $(0,m)$ 
are being generated as long as there results $y \leq f_1(x)$. At such stage, the corresponding 
value of $x$ is yielded as output (Step 4). We point out that Step 3 and Step 4 rely 
on the following proposition whose straightforward proof is omitted. 
\begin{proposition}
If $W_k(0)=x$, for all $\theta\in(0,\tau)$ there holds
$$
 \Prob\left(T_{k}^W\leq\theta\,\Big|\,T_{k}^W\leq\tau\right)
 =\Prob[X\leq\theta\,|\,Y\leq f_1(X)],
$$
where $X$ is uniformly distributed in $(0,\tau)$, $Y$ is uniformly
distributed in $(0,m)$, $X$ and $Y$ are independent, and $m$ is the
maximum over $(0,\tau)$ of the pdf $f_1(\theta)$ given in\eq{14}.
\end{proposition}
%
%
%
\subsection*{A2. Simulation when first passage has not yet occurred}
In the procedure listed below, again a classical VNAR method is implemented in 
order to construct a pseudo-random number from the r.v.\ having pdf\eq{15}: 
$$
 \framebox{
 $\begin{array}{l}
 \hbox{\bf Procedure}\\
 \hbox{\textsc{Step 1.}$\;$ by {\bf Subroutine} generate a pseudo-random number $z$ from the r.v.\ with pdf}\\
        \hbox{\qquad $f_3(z)$; \quad [see Eq.\eq{16}]}\\
 \hbox{\textsc{Step 2.}$\;$ generate an uniform pseudo-random number $u$ in $(0,1)$;}\\
 \hbox{\textsc{Step 3.}$\;$ if $u\leq\ell(z)$ then $x:=z$, else goto \textsc{Step 1};}\\
 \hbox{\textsc{Step 4.}$\;$ return($x$).}\\
 \hbox{{\bf Subroutine} (simulates a truncated Gaussian density)}\\
 \hbox{\textsc{Step S1.}$\;$ $q:=1-\Phi\left({\beta-x-\mu_k\,\tau \over \sigma_k\,\sqrt{\tau}}\right)$;}\\
 \hbox{\textsc{Step S2.}$\;$ if $q<\varepsilon_3$ then $\{$ generate a standard normal pseudo-random number $n$;}\\
   \hbox{\qquad $w:=n$ and goto \textsc{Step~S7} $\}$;}\\
 \hbox{\textsc{Step S3.}$\;$ generate two uniform independent 
pseudo-random numbers $u$ and $v$ in $(0,1)$;}\\
 \hbox{\textsc{Step S4.}$\;$ $y:=\log v+{\beta-x-\mu_k\,\tau \over \sigma_k\sqrt{\tau}}$;}\\
 \hbox{\textsc{Step S5.}$\;$ if $u\geq\xi(y)$ then goto \textsc{Step~S3}; \quad [see Eq.\eq{19}] }\\
 \hbox{\textsc{Step S6.}$\;$ $w:=y$;}\\
 \hbox{\textsc{Step S7.}$\;$ $z:=x+\mu_k\,\tau+w\,\sigma_k\,\sqrt{\tau}$;}\\
 \hbox{\textsc{Step S8.}$\;$ return($z$).}
 \end{array}$ 
}
$$
We point out that use of VNAR method requires that the density\eq{15} to be 
simulated be factorized as follows: 
\begin{equation}
 f_2(z)=C\,f_3(z)\,\ell(z),
 \qquad -\infty<z<\beta,
 \label{equation:18}
\end{equation}
where $C=[1-G_{k}(\beta,\tau\,|\,x)]^{-1}$,
\begin{equation}
 f_3(z)={1\over\sqrt{2\pi\sigma^2_k \tau}}\,
 \exp\left\{-{(z-x-\mu_k\tau)^2\over 2\sigma^2_k \tau}\right\}
 {1\over\Phi\left({\beta-x-\mu_k\tau \over \sigma_k\sqrt{\tau}}\right)},
 \qquad -\infty<z<\beta,
 \label{equation:16}
\end{equation}
and
\begin{equation}
 \ell(z)=\Phi\left({\beta-x-\mu_k\tau \over \sigma_k\sqrt{\tau}}\right)
 \left[1-\exp\left\{-{2\over\sigma^2_k\tau}\,(\beta-z)\,(\beta-x)\right\}\right],
 \qquad -\infty<z<\beta.
 \label{equation:17}
\end{equation}
Note that factorization\eq{16} is made possible by virtue of\eq{13}. 
Step 1 of the procedure can thus be performed. 

We note that $C\geq 1$, $f_3(z)$ is the truncation over $(-\infty,\beta)$ 
of a Gaussian density with mean $x+\mu_k\tau$ and variance $\sigma^2_k\tau$, 
while $\ell(z)\in[0,1]$ for all $z\in(-\infty,\beta)$. The implemented 
VNAR method requires the generation by a suitable subroutine of a 
pseudo-random number from the r.v.\ $Z$ having pdf\eq{16} as summarized 
in Steps S1-S8. Step S7 gives $Z$ as output when $U\leq\ell(Z)$, where 
$U$ is an uniform pseudo-random number in $(0,1)$. We point out that 
the described procedure makes use of the following  
\begin{proposition}
For all $z\in(-\infty,\beta)$, we have: 
$$
 \int_{-\infty}^z f_2(y)\,dy
 =\Prob\left[Z\leq z\,|\,U\leq\ell(Z)\right],
$$
where $U$ and $Z$ are independent r.v.'s that are distributed
uniformly in $(0,1)$ and according to pdf\eq{16}, respectively.
\end{proposition}
%
%
%
\par
The above subroutine is based on the following considerations. First of all, 
the subroutine generates a pseudo-random number $z$ from the r.v.\ with pdf 
$f_3(z)$ given in Eq.\eq{16}. We note that if\eq{16} is the pdf of $Z$, then 
\begin{equation}
 W= \frac{Z-x-\mu_k \tau}{\sigma_k \sqrt{\tau}}
 \label{equation:2}
\end{equation}
is a truncated standard normal r.v.\ possessing pdf
\begin{equation}
 f_W(z)=\ds{\frac{1}{\sqrt{2 \pi}}}\exp\left\{-\frac{z^2}{2}\right\}
 \frac{1}{\Phi\left(\frac{\beta-x-\mu_k\,\tau}{\sigma_k\,\sqrt{\tau}}\right)},
 \qquad z<\frac{\beta-x-\mu_k\,\tau}{\sigma_k\,\sqrt{\tau}}.
 \label{equation:21}
\end{equation}
Via\eq{2} the subroutine thus generates a pseudo-random number from $Z$ by 
sampling a pseudo-random number from $W$. To this purpose, for a preassigned 
tolerance parameter $\varepsilon_3$, if
\begin{equation}
 1-\Phi\left(\frac{\beta-x-\mu_k\,\tau}{\sigma_k\,\sqrt{\tau}}\right)
 <\varepsilon_3,
 \label{equation:22}
\end{equation}
then the distribution of $W$ is approximated by a standard normal distribution. 
Therefore, a generation of a value for $W$ is performed by implementing any
classical generator of standard gaussian r.v.\ (see Step S2). If\eq{22} is not 
fulfilled, the following classical VNAR method is employed (see Steps S3-S6). 
Consider the r.v.\ $Y$, having pdf
\begin{equation}
 f_4(y)=\cases{
 \exp\left\{y-\ds{\beta-x-\mu_k\,\tau \over \sigma_k\,\sqrt{\tau}}\right\}
 & if $y<\ds{\beta-x-\mu_k\,\tau \over \sigma_k\,\sqrt{\tau}}$, \\
 \hfill\cr
 0 & otherwise,}
 \label{equation:20}
\end{equation}
with $k\in\{1,2\}$, $x,\beta\in\Reali$, $x\leq\beta$ and $\tau>0$.
Then, the pdf given in\eq{21} can be expressed as
$$
 f_W(z)=D\,f_4(z)\,\xi(z),
 \qquad
 z<\frac{\beta-x-\mu_k\,\tau}{\sigma_k\,\sqrt{\tau}},
$$
where
$$
 D=\frac{1}{\Phi
 \left(\frac{\beta - x - \mu_k \, \tau}{\sigma_k \, \sqrt{\tau}} \right)}
 \exp\left\{\frac{\beta-x-\mu_k \, \tau}{\sigma_k \sqrt{\tau}} + 4\right\}
$$
and
\begin{equation}
 \xi(z)=\frac{1}{\sqrt{2 \pi}} \, \exp\left\{-\frac{z^2 + 2 \, z + 8}{2}\right\},
 \label{equation:19}
\end{equation}
with $0\leq\xi(z)\leq 1$. Note that $D\geq 1$ since
$$
 \frac{\beta-x-\mu_k\,\tau}{\sigma_k\,\sqrt{\tau}}+4>0.
$$
A pseudo-random number from $Y$ is generated (Step S4) by using the inversion 
method of cdf's. We thus have
$$
 Y=\log U_{2}+\frac{\beta-x-\mu_k\,\tau}{\sigma_k\sqrt{\tau}}
 \qquad \Longleftrightarrow \qquad
 U_2=e^{Y-\ds{\frac{\beta-x-\mu_k\,\tau}{\sigma_k\sqrt{\tau}}}}
 =F_Y(Y),
$$
where $F_Y(\cdot)$ is the cdf of $Y$. Then, (Step S5), the subroutine samples 
a value for $W$ as long as
$$
 u \leq \xi \left( \log v + \frac{\beta-x-\mu_k \, \tau}{\sigma_k \sqrt{\tau}}\right),
$$
with $u$ and $v$ pseudo-random numbers in $(0,1)$ and where $\xi(\cdot)$ is given 
in \eq{19}. The theoretical justification is provided by the following proposition, 
whose straightforward proof is omitted. 
\begin{proposition}
For all $z<(\beta-x-\mu_k\,\tau)/(\sigma_k\,\sqrt{\tau})$ it is
$$
 \int_{-\infty}^z f_W(y)\,dy=\Prob[Y<z\,|\,U<\xi(Y)],
$$
where $f_W(y)$ is the pdf given in\eq{21}, $U$ and $Y$ are independent
r.v.'s such that $U$ is uniformly distributed in $(0,1)$ and $Y$ has pdf\eq{20},
and where $\xi(\cdot)$ is defined in\eq{19}.
\end{proposition}
%
%
\par
The theoretical justification of the simulation procedure is thus completed. 
\subsection*{\bf Acknowledgments}
\setlength{\baselineskip}{13pt}
We thank an anonymous reviewer for his criticism on a previous version of this paper. 
This work has been performed within a joint cooperation agreement between
Japan Science and Technology Corporation (JST) and Universit\`a di Napoli
Federico II, under partial support by MIUR (cofin 2003) and by G.N.C.S. (INdAM).
%
\subsection*{\bf References}
%
\begin{description}
\item
A. Buonocore, A. Di Crescenzo and E. Di Nardo, 
``Input-output behavior of a model neuron with alternating drift,''
{\em BioSystems\/} vol. 67 pp. 27--34, 2002. 
\item
A. Buonocore, A.G. Nobile and L.M. Ricciardi, 
``A new integral equation for the evaluation of first-passage-time 
probability densities,''
{\em Adv. Appl. Prob.} vol. 27 pp. 102--114, 1987.
\item
D. Cyranoski, 
``Swimming against the tide,'' 
{\em Nature} vol. 408 pp. 764--766, 2000. 
\item
A. Di Crescenzo, 
``On Brownian motions with alternating drifts,'' 
in {\it Cybernetics and Systems 2000\/} (Trappl R. ed.), 
Vienna, Austria, 2000, pp. 324--329. 
\item
E. Di Nardo, A.G. Nobile, E. Pirozzi and L.M. Ricciardi, 
``A computational approach to first-passage-time problems 
for Gauss-Markov processes,'' 
{\em Adv. Appl. Prob.} vol. 33 pp. 453--482, 2001.
\item
M. Freidlin and H. Pavlopoulos, 
``On a stochastic model for moisture budget in an Eulerian atmospheric column,''
{\em Environmetrics\/} vol. 8 pp. 425--440, 1997.
\item
A. Giorno, A.G. Nobile and L.M. Ricciardi, 
``On the evaluation of first-passage-time probability densities via nonsingular equations,''
{\em Adv. Appl. Prob.} vol. 21 pp. 20--36, 1989. 
\item
M. Greiner, M. Jobmann and C. Kl\"{u}ppelberg, 
``Telecommunication traffic, queueing models, and subexponential distributions,''
{\em Queueing Systems} vol. 33 pp. 125--152, 1999. 
\item
K. Kitamura, M. Tokunaga, A. Hikikoshi Iwane  and T. Yanagida,   
``A single myosin head moves along an actin filament with regular steps of 5.3 nanometres,'' 
{\em Nature} vol. 397 pp. 129--134, 1999.
\item
J.R. Michael, W.R. Schucany and R.W. Haas, 
``Generating random variates using transformations with multiple roots,'' 
{\it The American Statistician\/} vol. 30 pp. 88--90, 1976.
\item
T. Mikosch and A.V. Nagaev, 
``Large deviations of heavy-tailed sums with applications in insurance,'' 
{\it Extremes\/} vol. 1 pp. 81--110, 1998.
\item
L.M. Ricciardi, A. Di Crescenzo, V. Giorno and A.G. Nobile, 
``An outline of theoretical and algorithmic approaches to first
passage time problems with applications to biological modeling,''
{\it Mathematica Japonica\/} vol. 50 pp. 247--322, 1999.
\item
S. Ross, 
{\em Introduction to Probability Models\/}, Fourth Edition, 
Academic Press: Boston, 1989.
\item
M. Shaked and J.G. Shanthikumar, 
{\em Stochastic Orders and Their Applications\/}, 
Academic Press: San Diego, 1994.
\item
B.W. Silverman, 
{\em Density Estimation for Statistics and Data Analysis\/}, 
Chapman and Hall: London, 1986. 
\end{description}
\newpage

\begin{figure} 
\centering{
\epsfxsize=12cm
\epsfysize=8cm
\epsfbox{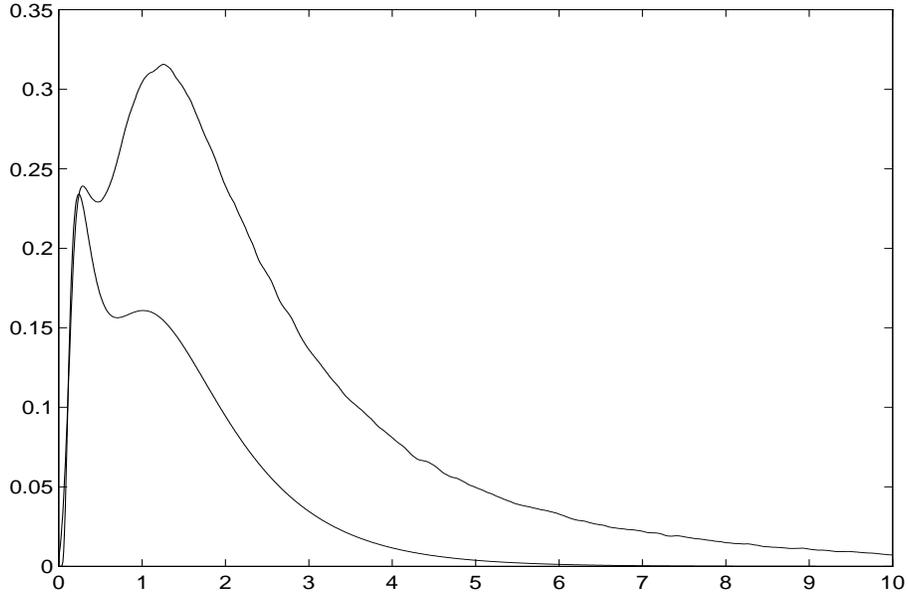}
\caption{\small 
FPT density $h_{k}(\beta,t\,|\,x_0)$ and its lower bound for exponentially 
distributed alternating times, with rates $\lambda_1=1/2$ and $\lambda_2=2$, and 
$\mu_1=1$, $\mu_2=-1$, $\sigma_1^2=1$, $\sigma_2^2=10$, $\beta=3$, $x_0=0$ and $k=2$.}
}
\end{figure}
\begin{figure} 
\centering{
\epsfxsize=12cm
\epsfysize=8cm
\epsfbox{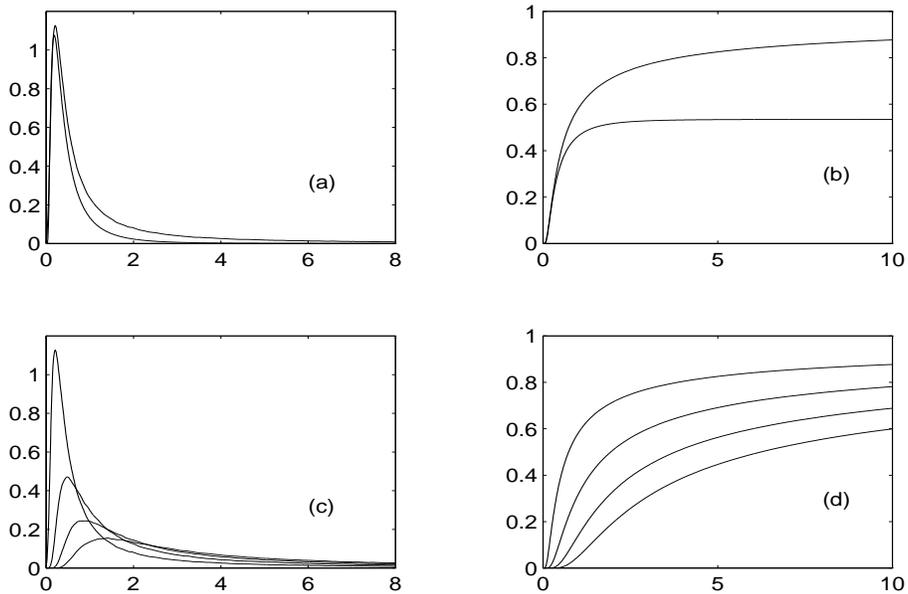}
\caption{\small 
(a) FPT density $h_{k}(\beta,t\,|\,x_0)$ and its lower bound for inverse Gaussian 
distributed alternating times, with parameters given in\eq{31}, $k=1$, $x_0=0$ and $\beta=0.3$; 
(b) the corresponding distribution function $H_{k}(\beta,t\,|\,x_0)$ and its lower bound; 
(c) FPT densities as in (a) with (from top to bottom near the origin) $\beta=0.3$, $0.5$, $0.7$, $0.9$;  
(d) the corresponding distribution functions.}
}
\end{figure}
\begin{figure} 
\centering{
\epsfxsize=12cm
\epsfysize=8cm
\epsfbox{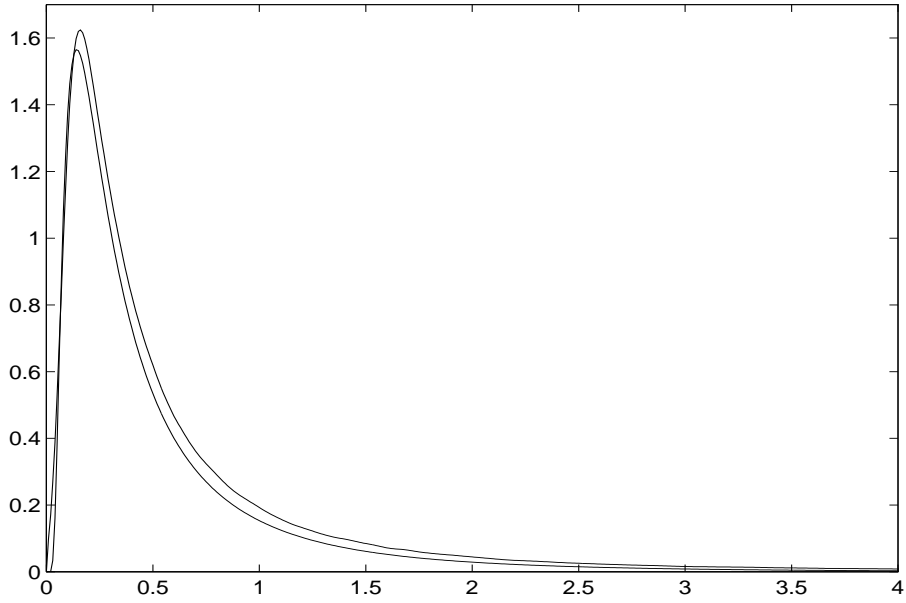}
\caption{\small 
An example of estimated FPT density $h_{1}^S(2,t\,|\,1)$ and of its lower bound. Here process\eq{34} is 
characterized by $\mu_1=1$, $\mu_2=-1$ and $\sigma_1=\sigma_2=1$. The alternating times occur according 
to a generalized Pareto distribution\eq{41}, with $\xi_1=1$, $\eta_1=2$, $\xi_2=3$ and $\eta_2=4$.}
}
\end{figure}
\begin{figure} 
\centering{
\epsfxsize=12cm
\epsfysize=8cm
\epsfbox{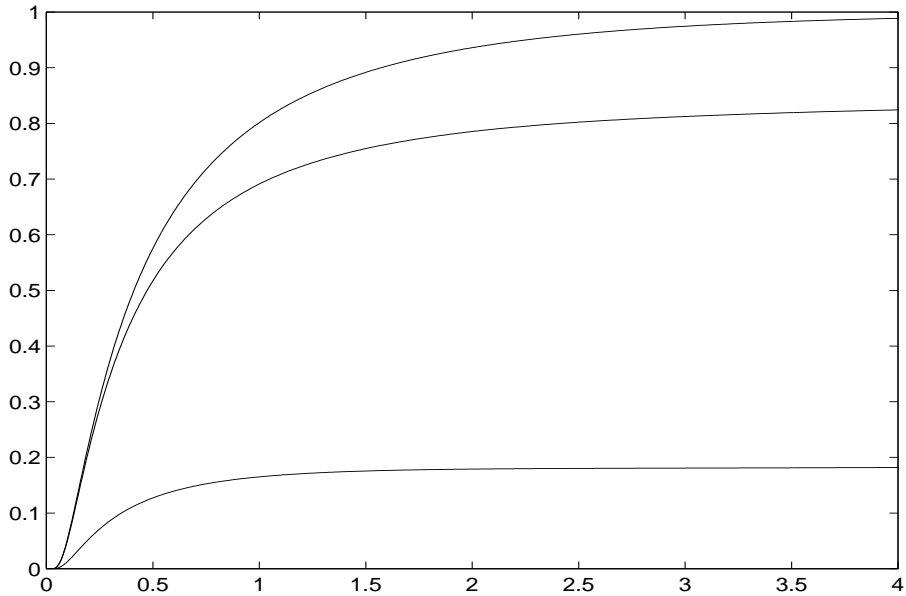}
\caption{\small 
Cumulative distribution function $H_{1}^S(2,t\,|\,1)$ and its bounds under the assumptions of Fig.\ 3. 
The upper bound is obtained from\eq{37}, holding for $\sigma_1=\sigma_2$.}
}
\end{figure}
%
\end{document}